\documentclass[12pt]{article}
\usepackage{ae} 
\usepackage{aecompl} 

\usepackage{amsmath}
\usepackage{amsthm}
\usepackage{amsfonts}
\usepackage{amssymb}
\usepackage{stmaryrd}
\usepackage{mathtools}
\usepackage{bm}
\usepackage{bbm}
\usepackage{tikz}
\usetikzlibrary{calc}
\usetikzlibrary{arrows}
\usepackage{subcaption}
\usepackage{enumitem}

\usepackage[%
  pdftex,
  plainpages=false,
  pdfpagelabels,
  colorlinks,
  citecolor=black,
  linkcolor=black,
  urlcolor=black,
  filecolor=black,
  bookmarksopen=false
]{hyperref} 
\usepackage[all]{hypcap} 

\def\({\left(}
\def\){\right)}
\def\[{\left[}
\def\]{\right]}
\def\<{\left\langle}
\def\>{\right\rangle}
\def\lv{\left\lvert}

\def\rv{\right\rvert}

\def\llb{\left\llbracket}
\def\rrb{\right\rrbracket}

\def\iff{\Longleftrightarrow}
\def\implies{\Longrightarrow}

\let\geq\geqslant
\let\leq\leqslant
\let\:\colon

\let\epsilon\varepsilon
\let\oldphi\phi
\let\phi\varphi

\def\NN{\ensuremath{\mathbb{N}}}

\def\PP{\ensuremath{\mathbb{P}}}

\def\RR{\ensuremath{\mathbb{R}}}

\def\ZZ{\ensuremath{\mathbb{Z}}}

\def\One{\ensuremath{\mathbbm{1}}}

\def\cA{\ensuremath{\mathcal{A}}}

\def\cC{\ensuremath{\mathcal{C}}}

\def\cF{\ensuremath{\mathcal{F}}}

\def\cW{\ensuremath{\mathcal{W}}}

\let\phi\oldphi

\DeclareMathOperator{\Tr}{Tr}
\DeclareMathOperator{\Aut}{Aut}

\newcommand{\edge}[2]{\langle #1,#2 \rangle}
\DeclareMathOperator{\Hom}{Hom}

\newcommand{\expect}[1]{ {\mathbf E}\! \left[ #1 \right] }

\newcommand{\phiqr}{\phi_{\text{qr}}}

\renewcommand{\One}{1}

\newtheoremstyle{plain}
  {\medskipamount}
  {\smallskipamount}
  {\slshape}
  {0pt}
  {\bfseries}
  {.}
  { }
  {\thmname{#1}\thmnumber{ #2}{\normalfont\thmnote{ (#3)}}}

\theoremstyle{plain}

\newtheorem{theorem}{Theorem}[section]
\newtheorem{proposition}[theorem]{Proposition}
\newtheorem{lemma}[theorem]{Lemma}
\newtheorem{corollary}[theorem]{Corollary}
\newtheorem{conjecture}[theorem]{Conjecture}
\newtheorem{remark}[theorem]{Remark}
\newtheorem*{remark*}{Remark}
\newtheorem*{example*}{Example}

\title{Quasi-Carousel Tournaments}
\author{Leonardo Nagami Coregliano\thanks{Instituto de Matem\'atica e Estat\'istica,
    Universidade de S\~ao Paulo, \nolinkurl{lenacore@ime.usp.br}. Work done while
    visiting University of Chicago, supported by Funda\c c\~ao de Amparo \`a Pesquisa
    do Estado de S\~ao Paulo (FAPESP) under grants no.~2013/23720-9
    and~2014/15134-5.}}

\begin{document}
\maketitle

\begin{abstract}
  A tournament is called locally transitive if the outneighbourhood and the
  inneighbourhood of every vertex are transitive. Equivalently, a tournament is
  locally transitive if it avoids the tournaments~$W_4$ and~$L_4$, which are the only
  tournaments up to isomorphism on four vertices containing a unique~$3$-cycle. On
  the other hand, a sequence of tournaments~$(T_n)_{n\in\NN}$ with~$\lv V(T_n)\rv =
  n$ is called almost balanced if all but~$o(n)$ vertices of~$T_n$ have
  outdegree~$(1/2 + o(1))n$. In the same spirit of quasi-random properties, we
  present several characterizations of tournament sequences that are both almost
  balanced and asymptotically locally transitive in the sense that the density
  of~$W_4$ and~$L_4$ in~$T_n$ goes to zero as~$n$ goes to infinity.
\end{abstract}

A \emph{balanced} tournament~$T$ is a tournament with an odd number of
vertices~$2n+1$ such that every vertex of~$T$ has outdegree~$n$. On the other hand, a
\emph{locally transitive} tournament~$T$ is a tournament such that the
outneighbourhood~$N^+(v) = \{w\in V(T) : vw\in A(T)\}$ and the
inneighbourhood~$N^-(v) = \{w\in V(T) : wv\in A(T)\}$ of every vertex~$v$ are both
transitive. With these definitions, there is only one up to isomorphism\footnote{This
  is a direct consequence of a result of
  Brouwer~\cite{B:TheEnumerationOfLocallyTransitiveTournaments}, which is on
  Section~\ref{sec:loctran} of this paper.} balanced locally transitive
tournament~$R_{2n+1}$ (see Figure~\ref{fig:R}) of order~$2n+1$ for each~$n\in\NN$,
which we call the \emph{carousel tournament}\footnote{This is because in a
  carousel, each horse is beating half of the other horses in a circular
  structure.} of order~$2n+1$. This tournament
is given by
\begin{align*}
  V(R_{2n+1}) & =\{0,1,\ldots,2n\}; &
  A(R_{2n+1}) & = \{(x,(x+i) \bmod (2n+1)) : i\in[n]\}\};
\end{align*}
where~$[n]=\{1,2,\ldots,n\}$ (and~$[0]=\varnothing$).

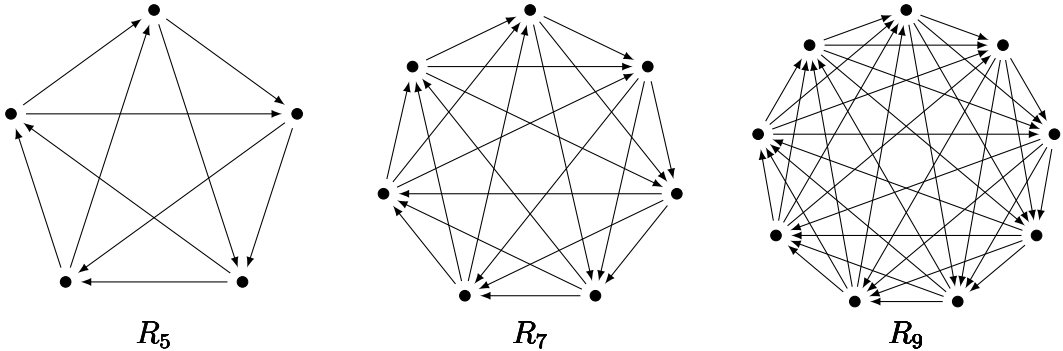
\begin{figure}[ht]
  \begin{center}
    \begin{tikzpicture}
\begingroup
\def\radius{2cm}
\def\shorten{0.2cm}
\def\step{5cm}
\foreach \n in {2,3,4}{
  \pgfmathsetmacro{\tn}{2*\n}
  \pgfmathsetmacro{\N}{\tn+1}
  \foreach \i in {0,...,\tn}{
    \pgfmathsetmacro{\angle}{\i*360/\N+90}
    \filldraw ($(\angle:\radius) + (\n*\step,0cm)$) circle (2pt);

    \foreach \j in {1,...,\n}{
      \pgfmathsetmacro{\angletwo}{\angle-\j*360/\N}
      \draw[shorten >=\shorten,shorten <=\shorten,arrows={-latex}]
      ($(\angle:\radius) + (\n*\step,0cm)$) --
      ($(\angletwo:\radius) + (\n*\step,0cm)$);
    }

    \node[below] at (\n*\step,-\radius) {$R_{\pgfmathprintnumber{\N}}$};
  }
}

\endgroup
\end{tikzpicture}
    \caption{The tournament~$R_{2n+1}$ for~$n=2,3,4$.}
    \label{fig:R}
  \end{center}
\end{figure}

Given the well-organized structure of the carousel tournaments, it is natural to
expect nice asymptotic properties to hold for the sequence~$(R_{2n+1})_{n\in\NN}$
and in this note we begin studying this sequence asymptotically in two directions. In
the first direction, we are simply interested in what are the asymptotic properties
of~$(R_{2n+1})_{n\in\NN}$. But, in the more stimulating second direction, we are
interested in the question: when does a sequence of tournaments~$(T_n)_{n\in\NN}$
``look like'' the sequence~$(R_{2n+1})_{n\in\NN}$?

Although this seems a rather vague question, it turns out that there is a notion of
similarity of sequences of combinatorial objects that yields a very rich field of
study. Namely, we say that two sequences of tournaments~$(T_n)_{n\in\NN}$
and~$(T_n')_{n\in\NN}$ are equivalent if for every fixed tournament~$T$ the density
of~$T$ in~$T_n$ is asymptotically equal to the density of~$T$ in~$T_n'$, that is, we
have
\begin{align*}
  \lim_{n\to\infty}p(T,T_n) - p(T,T_n') & = 0,
\end{align*}
where~$p(T,U)$ denotes the unlabelled density of~$T$ as a subtournament of~$U$.

This notion of similarity can be traced back to the theory of quasi-randomness,
originated with the study of graphs sequences (by comparing with the sequence of
Erd\H{o}s--R\'{e}nyi graphs~$(\bm{G_{n,1/2}})_{n\in\NN}$) in the seminal papers by
Thomason~\cite{T:PseudoRandomGraphs} and Chung, Graham,
Wilson~\cite{CGW:QuasiRandomGraphs} (see~\cite{KS:PseudoRandomGraphs} for a survey)
and now a field with branches in several combinatorial objects such as uniform
hypergraphs~\cite{CG:QuasiRandomHypergraphs, C:QuasiRandomHypergraphsRevisited,
  BR:NoteOnUpperDensityOfQuasiRandomHypergraphs}, graph
orientations~\cite{G:QuasiRandomOrientedGraphs},
permutations~\cite{C:QuasirandomPermutations,
  KP:QuasirandomPermutationsAreCharacterized} and
tournaments~\cite{CG:QuasiRandomTournaments,KS:ANoteOnEvenCyclesAndQuasirandomTournaments,
  CR:OnTheDensityOfTransitiveTournaments}.

Such notion of similarity also yields a very useful notion of convergence, namely, we
say that~$(T_n)_{n\in\NN}$ is convergent if~$(\lv V(T_n) \rv)_{n\in\NN}$ is
increasing and~$(p(T,T_n))_{n\in\NN}$ is convergent for every tournament~$T$. With
this notion of convergence, one can define limit objects that codify these
densities. One approach is to define the limit object to be semantically close to the
underlying combinatorial objects, that is, to find a limit object that resembles the
definition of the combinatorial objects (such approach was originated with the
definition of graphons~\cite{LS:LimitsOfDenseGraphSequences} and has also been taken
in the definition of hypergraphons~\cite{ES:MeasureTheoreticApproach},
permutons~\cite{HKMRS:LimitsOfPermutationSequences} and
digraphons~\cite[Section~9]{DJ:GraphLimitsAndExchangeableRandomGraphs}). Another
approach is to study the limit object syntactically, that is, to see what kind of
properties the sequence~$(\phi(T))_T$ must satisfy if we have~$\phi(T) =
\lim_{n\to\infty}p(T,T_n)$. This latter approach is precisely the thrust of the
theory of flag~algebras~\cite{R:FlagAlgebras} and in what follows, we will mostly use
this language.

In the particular case of quasi-random tournaments, we are interested in comparing
with the sequence~$(\bm{T_{1/2}(n)})_{n\in\NN}$, where~$\bm{T_{1/2}(n)}$ is the
random tournament of order~$n$ where each arc orientation is picked independently at
random with probability~$1/2$. It is a straightforward exercise on distribution
concentration to prove that~$(\bm{T_{1/2}(n)})_{n\in\NN}$ is a convergent sequence
with probability~$1$ and we call its limit~$\phiqr\in\Hom^+(\cA^0,\RR)$ in the
flag~algebra language the \emph{quasi-random homomorphism}. It is also
straightforward to prove that the sequence of carousel
tournaments~$(R_{2n+1})_{n\in\NN}$ is convergent\footnote{In Section~\ref{sec:conv},
  we also offer an alternative proof of this convergence that does not involve
  computing the limit of the densities~$p(T,R_{2n+1})$.} and we call its
limit~$\phi_R\in\Hom^+(\cA^0,\RR)$ the \emph{carousel homomorphism}.

The theory of quasi-random tournaments was inaugurated by Chung and~Graham
in~\cite{CG:QuasiRandomTournaments}, where they presented not only some quasi-random
tournament properties (their~$P$ properties), but also showed another class of
properties (their~$Q$ properties) that were equivalent to each other but were
strictly weaker than the quasi-random properties.

For every~$k\in\NN$, let~$\Tr_k$ denote the transitive tournament of order~$k$
and~$\vec C_3$ denote the directed~$3$-cycle. We are particularly interested in the
following~$Q$ properties of a sequence of tournaments~$(T_n)_{n\in\NN}$ with~$\lv
V(T_n)\rv = n$.
\begin{itemize}
\item $Q_1$: $\lim_{n\to\infty}p(\Tr_3,T_n) = 3/4$ and~$\lim_{n\to\infty}p(\vec C_3,
  T_n) = 1/4$;
\item $Q_2$: $p(\vec C_3,T_n)$ is asymptotically maximized by the
  sequence~$(T_n)_{n\in\NN}$;
\item $Q_3$: The sequence of tournaments~$(T_n)_{n\in\NN}$ of increasing orders is
  \emph{almost balanced}, that is, all but~$o(n)$ vertices of~$T_n$ have
  outdegree~$(1/2 + o(1))n$.
\end{itemize}

Now, consider the extremal problem of minimizing the density of a fixed
tournament~$T$ asymptotically in a sequence of tournaments~$(T_n)_{n\in\NN}$ of
increasing orders. In the language of flag~algebras, this can be cleanly stated as
minimizing~$\phi(T)$ for~$\phi\in\Hom^+(\cA^0,\RR)$.

If~$T$ is non-transitive, this problem is trivial because we can take~$T_n$ to be the
transitive tournament~$\Tr_n$ of size~$n$ and we will have~$p(T,T_n)=0$ for
every~$n\in\NN$.

For the transitive case, Chung and~Graham's~\cite{CG:QuasiRandomTournaments}
property~$Q_2$ implies that~$\phi(\Tr_3)$ is minimized if and only if~$\phi$ is the
limit of an almost balanced sequence. Later,
Griffiths~\cite{G:QuasiRandomOrientedGraphs} proved that~$\phi(\Tr_4)$ is minimized
if and only if~$\phi$ is the quasi-random homomorphism~$\phiqr$. Finally, the
minimization problem for a single tournament was closed when Griffith's result was
extended in~\cite{CR:OnTheDensityOfTransitiveTournaments}: for~$k\geq 4$, the
density~$\phi(\Tr_k)$ is minimized if and only if~$\phi$ is the quasi-random
homomorphism~$\phiqr$.

Now, if we consider the analogous maximization problem, the gears completely reverse:
the transitive case becomes the trivial case (since~$p(\Tr_k,\Tr_n) = 1$ for
every~$k\leq n$) and property~$Q_2$ of Chung and~Graham says that~$\phi(\vec C_3)$ is
maximized if and only if~$\phi$ is the limit of an almost balanced sequence. However,
this leaves the maximization problem open for every non-transitive tournament of
order at least~$4$, thus making the maximization problem much more meaningful.

In this note, we begin studying this maximization problem by proving that for the
unique tournament~$R_4$ with outdegree sequence~$(1,1,2,2)$, the density~$\phi(R_4)$
is maximized if and only if~$\phi$ is the carousel
homomorphism~$\phi_R$. Furthermore, in the same spirit of the quasi-randomness
theory, we present several properties that a sequence of tournaments has if and only
if it is equivalent to~$(R_{2n+1})_{n\in\NN}$ (and we call a sequence having these
properties a \emph{quasi-carousel sequence}).

In the same flavour of the carousel tournaments, one of these properties
implies that~$\phi_R$ is the only balanced locally transitive homomorphism after we
extend the notions of balancedness and local transitivity to homomorphisms.

Let us also highlight another set of properties of the carousel homomorphism~$\phi_R$
that have nice analogues for the quasi-random homomorphism~$\phiqr$. If~$\edge uv$ is
an arc of a tournament~$T$, all other vertices~$w\in V(T)\setminus\{u,v\}$ can be
classified into four classes (``flags''):
\begin{enumerate}
\item $\edge uw, \edge vw\in E(T)$,

\item $\edge wu, \edge wv\in E(T)$,

\item $\edge uw, \edge wv\in E(T)$.

\item $\edge vw, \edge wu\in E(T)$,
\end{enumerate}

Following and expanding a bit the notation
in~\cite{R:OnTheCaccettaHaggkvistConjecture}, we let~$O^A(u,v)$, $I^A(u,v)$,
$\Tr_3^A(u,v)$ and~$\vec C_3^A(u,v)$ denote the numbers of vertices in the four
classes (taken in this order, see also Figure~\ref{fig:typesandflags}) divided
by~$|V(T)|-2$. A set of interesting characterizations of quasi-randomness says that
if~$F$ is any of~$O^A$, $I^A$, $\Tr_3^A$ or~$\vec C_3^A$, then a sequence of
tournaments~$(T_n)_{n\in\NN}$ is quasi-random if and only if~$F(u,v)$ is
``nearly''~$1/4$ for ``almost all'' arcs~$\edge uv$ (the theorems for~$O^A$
and~$I^A$ are from~\cite{CG:QuasiRandomTournaments} and the theorems for the other
two classes are from~\cite{CR:OnTheDensityOfTransitiveTournaments}). This can be
stated formally and cleanly\footnote{But can be stated even more cleanly in the
  language of flag~algebras using extensions of homomorphisms~\cite[\S
    3.2]{R:FlagAlgebras}.} by saying that if~$\bm{\edge{u_n}{v_n}}$ is a random arc
of~$T_n$ picked uniformly at random, then the sequence of random
variables~$(F(\bm{u_n},\bm{v_n}))_{n\in\NN}$ converges almost surely to~$1/4$.

In the case of the carousel homomorphism, we prove an interesting analogous
characterization: a sequence~$(T_n)_{n\in\NN}$ converges to~$\phi_R$ if and only if
the sequence of random variables~$F(\bm{u_n},\bm{v_n})$ converges in distribution to
the uniform random variable on~$[0,1/2]$.

The note is organized as follows. In Section~\ref{sec:loctran}, we review some basic
properties of locally transitive tournaments. In Section~\ref{sec:flag}, we remind
some concepts of the theory of flag~algebras and of the tournament quasi-randomness
theory. We also establish some basic lemmas on the flag~algebra of tournaments in the
same section. In Section~\ref{sec:phiR}, we present the main theorem that
characterizes the carousel homomorphism~$\phi_R$, but we defer the
proof of convergence of the sequence~$(R_{2n+1})_{n\in\NN}$ to
Section~\ref{sec:conv}. Finally, in Section~\ref{sec:conc}, we present some related
open problems.

\section{Locally Transitive Tournaments}
\label{sec:loctran}

In this section, we remind some basic properties of locally transitive tournaments.

A tournament~$T$ is called \emph{locally transitive} if for every vertex~$v\in V(T)$,
the outneighbourhood~$N^+(v) = \{w\in V(T) : vw\in A(T)\}$ and the
inneighbourhood~$N^-(v) = \{w\in V(T) : wv\in A(T)\}$ of~$v$ are both transitive.

Let~$W_4$ and~$L_4$ denote the (unique) tournaments with outdegree
sequences~$(1,1,1,3)$ and~$(0,2,2,2)$ respectively (i.e., these are precisely the
tournaments of order~$4$ that have a unique copy of a directed~$3$-cycle~$\vec
C_3$). The following characterization follows immediately from the definition of
local transitivity.

\begin{proposition}\label{prop:loctranW4L4}
  A tournament~$T$ is locally transitive if and only if~$T$ has no copy of~$W_4$ nor
  of~$L_4$.
\end{proposition}

Note that if~$v$ is a vertex of a locally transitive tournament~$T$, then the arcs
of~$T$ induce linear orders on~$N^+(v)$ and~$N^-(v)$ (that is, defining~$w <_T z \iff
wz\in A(T)$, the restriction of the relation~$<_T$ to either of these sets is a
linear order). With this observation, Brouwer obtained the following properties.

\begin{proposition}[Brouwer~\cite{B:TheEnumerationOfLocallyTransitiveTournaments}]
  \label{prop:cupint}
  If~$v$ is a vertex of a locally transitive tournament~$T$ and~$a\in N^+(v)$,
  then~$N^+(a)$ is the union of a terminal interval of~$N^+(v)$ and an initial
  interval of~$N^-(v)$ (in the order induced by the arcs of~$T$).
\end{proposition}

\begin{proof}
  From the order induced on~$N^+(v)$, it follows that~$N^+(a)\cap N^+(v)$ is a
  terminal interval of~$N^+(v)$. This means that if the proposition is false, there
  must exist~$b,c\in N^-(v)$ such that~$bc\in A(T)$, $c\in N^+(a)$ and~$b\notin
  N^+(a)$. This implies that~$a,c,v\in N^+(b)$ and~$ac,cv,va\in A(T)$, hence~$N^+(b)$
  is not transitive, a contradiction.
\end{proof}

\begin{proposition}[Brouwer~\cite{B:TheEnumerationOfLocallyTransitiveTournaments}]
  \label{prop:loctranchar}
  A tournament~$T$ is locally transitive if and only if it can be cyclically ordered
  in a way such that
  \begin{enumerate}[label=(\roman*)]
  \item For every vertex~$v\in V(T)$, the sets~$N^+(v)\cup\{v\}$
    and~$N^-(v)\cup\{v\}$ are intervals of the cyclic order (with one endpoint
    being~$v$);
    \label{it:cycint}
  \item For every vertices~$v,a\in V(T)$ with~$a\in N^+(v)$, the set~$N^+(a)$ is the
    union of a terminal interval of~$N^+(v)$ and an initial interval of~$N^-(v)$ (in
    the cyclic order).
    \label{it:cupint}
  \end{enumerate}
\end{proposition}

\begin{proof}
  Suppose~$T$ is a locally transitive tournament of order~$n$ and let~$w_0$ be one of
  its vertices. Let~$w_1,w_2,\ldots,w_k$ be the vertices in~$N^+(w_0)$ in the order
  induced by the arcs of~$T$ and let~$w_{k+1},w_{k+2},\ldots,w_{n-1}$ be the vertices
  in~$N^-(w_0)$ in the order induced by the arcs of~$T$.

  Consider the cyclic order induced by the mapping~$\ZZ_n\ni i\mapsto w_i\in V(T)$,
  where~$\ZZ_n=\ZZ/(n\ZZ)$ denotes the cyclic group of order~$n$.

  Trivially item~\ref{it:cycint} holds for~$v=w_0$. Note also that to prove
  item~\ref{it:cycint} for a vertex~$v$, it is enough to prove just the assertion
  regarding the set~$N^+(v)\cup\{v\}$.

  Now, since the orders on~$N^+(w_0)$ and~$N^-(w_0)$ induced by the arcs of~$T$
  coincide with the orders induced by the cyclic order defined, if~$v\in N^+(w_0)$,
  then Proposition~\ref{prop:cupint} implies that~$N^+(v)$ is of the form
  \begin{align*}
    \{w_i,w_{i+1},\ldots,w_k\}\cup\{w_{k+1},w_{k+2},\ldots,w_j\},
  \end{align*}
  for some~$i \leq j$, hence an interval of the cyclic order. Furthermore, the
  definition of the cyclic order implies that~$w_{i-1}=v$, hence is~$N^+(v)\cup\{v\}$
  an interval of cyclic order with one endpoint being~$v$.

  Finally, suppose that~$v\in N^-(w_0)$. From the definition of the cyclic order, we know
  that~$(N^+(v)\cup\{v\})\cap N^-(w_0)$ is an interval with endpoints~$v$
  and~$v_{n-1}$. On the other hand, Proposition~\ref{prop:cupint} implies
  that~$N^+(v)\cap N^-(w_0)$ must be a terminal interval of~$N^+(v)$ in the order
  induced by the arcs of~$T$, but since this order coincides with the one induced by
  the cyclic order in~$N^-(w_0)$, we have that~$N^+(v)\cup\{v\}$ is an interval with
  an endpoint being~$v$.

  Now that item~\ref{it:cycint} is proved, we know that for every vertex~$v\in V(T)$
  the order induced by the arcs of~$T$ in the sets~$N^+(v)$ and~$N^-(v)$ coincide
  with the ones induced by the cyclic order. With this observation,
  item~\ref{it:cupint} follows directly from Proposition~\ref{prop:cupint}.

  \medskip

  Suppose now that~$T$ is not locally transitive. By
  Proposition~\ref{prop:loctranW4L4}, there must be a set~$X$ of four vertices of~$T$
  that induces an occurrence of either~$W_4$ or~$L_4$ in~$T$.

  Note that any cyclic order satisfying items~\ref{it:cycint} and~\ref{it:cupint}
  in~$T$ must induce a cyclic order on~$X$ that satisfies these items in the
  tournament induced by this set.

  Since neither~$W_4$ nor~$L_4$ have a cyclic ordering satisfying both
  items~\ref{it:cycint} and~\ref{it:cupint}, the proof is complete.
\end{proof}

Recalling that a \emph{balanced} tournament is a tournament of odd order~$2n+1$ such
that every vertex has outdegree~$n$, we get the following corollary.

\begin{corollary}\label{cor:Runique}
  For every~$n\in\NN$, there is exactly one up to isomorphism balanced locally
  transitive tournament~$R_{2n+1}$ (see Figure~\ref{fig:R}) of order~$2n+1$ and it is
  given by
  \begin{align*}
    V(R_{2n+1}) & =\{0,1,\ldots,2n\}; &
    A(R_{2n+1}) & = \{(x,(x+i) \bmod (2n+1)) : i\in[n]\}\};
  \end{align*}
  where~$[n]=\{1,2,\ldots,n\}$ (and~$[0]=\varnothing$).
\end{corollary}

\begin{proof}
  Trivially~$R_{2n+1}$ is a balanced locally transitive tournament.

  On the other hand, if~$T$ is a balanced locally transitive tournament of
  order~$2n+1$, Proposition~\ref{prop:loctranchar} gives us a cyclic
  ordering~$f\:\ZZ_{2n+1}\to V(T)$, where~$\ZZ_{2n+1}=\ZZ/((2n+1)\ZZ)$ denotes the
  cyclic group of order~$2n+1$. It is easy to see that~$f$ is an isomorphism
  between~$R_{2n+1}$ and~$T$.
\end{proof}

We call~$R_{2n+1}$ the \emph{carousel tournament} of order~$2n+1$.

\begin{remark}\label{rmk:Rname}
  Although we define the carousel tournament~$R_n$ only for odd values of~$n$,
  our choice of notation~$R$ comes from analogy with the structure of~$R_4$, which is
  the locally transitive tournament of order~$4$ closest to being balanced.
\end{remark}

\section{Almost Balanced Tournament Sequences in Flag~Algebras}
\label{sec:flag}

In this section, we translate the results of the theory of quasi-random tournaments
regarding almost balanced tournament sequences to the language of flag~algebras. We
also add another characterization that will be useful later on. We assume the reader
has some familiarity with the basic setting of flag~algebras and with the notion of
extensions of homomorphisms~\cite[\S 3.2]{R:FlagAlgebras}.

Following the notation of~\cite{R:FlagAlgebras,R:OnTheCaccettaHaggkvistConjecture},
we consider the theory of tournaments~$T_{\text{Tournaments}}$ (and we will drop this
from notation when it is clear from the context). We let~$0$ denote the trivial type
of order~$0$ and~$1$ denote the (unique) type of order~$1$ as usual. We also
define~$A$ to be the type of order~$2$ such that the vertex labelled with~$1$ beats
the other (labelled) vertex (see Figure~\ref{fig:typesandflags}). For a
type~$\sigma$, we denote the unity of the algebra~$\cA^\sigma$ by~$\One_\sigma$, and,
as always, the element~$\One_0$ is abbreviated to~$\One$.

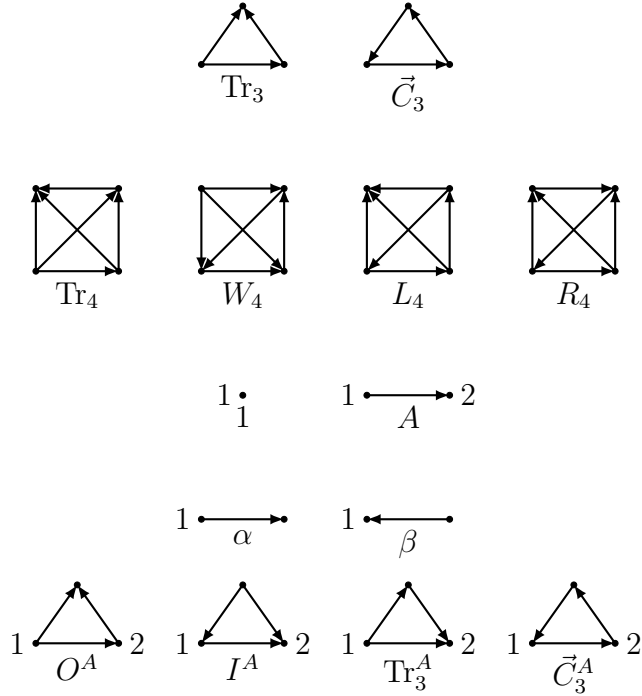
\begin{figure}[ht]
  \begin{center}
    \begin{tikzpicture}[scale=1.1]
  \foreach \i in {0,...,3}{
    \pgfmathsetmacro{\base}{2*\i}
    \pgfmathsetmacro{\next}{\base+1}

    \coordinate (T3A\i) at (\base cm, 1cm);
    \coordinate (T3B\i) at (\next cm, 1cm);
    \coordinate (T3C\i) at ($1/2*(T3A\i) + 1/2*(T3B\i) + (0cm,0.707106cm)$);

    \coordinate (T4A\i) at (\base cm,-1.5cm);
    \coordinate (T4B\i) at (\next cm,-1.5cm);
    \coordinate (T4C\i) at (\next cm,-0.5cm);
    \coordinate (T4D\i) at (\base cm,-0.5cm);

    \coordinate (TpA\i) at (\base cm, -3cm);
    \coordinate (TpC\i) at (\next cm, -3cm);
    \coordinate (TpB\i) at ($1/2*(TpA\i) + 1/2*(TpC\i)$);

    \coordinate (F2A\i) at (\base cm,-4.5cm);
    \coordinate (F2B\i) at (\next cm,-4.5cm);

    \coordinate (F3A\i) at (\base cm, -6cm);
    \coordinate (F3B\i) at (\next cm, -6cm);
    \coordinate (F3C\i) at ($1/2*(F3A\i) + 1/2*(F3B\i) + (0cm,0.707106cm)$);
  }

  \foreach \i in {1,2}{
    \filldraw (T3A\i) circle (1pt);
    \filldraw (T3B\i) circle (1pt);
    \filldraw (T3C\i) circle (1pt);

    \draw[thick, arrows={-latex}] (T3A\i) -- (T3B\i);
    \draw[thick, arrows={-latex}] (T3B\i) -- (T3C\i);
  }
  \draw[thick, arrows={-latex}] (T3A1) -- (T3C1);
  \draw[thick, arrows={-latex}] (T3C2) -- (T3A2);

  \node[below] at ($1/2*(T3A1) + 1/2*(T3B1)$) {$\Tr_3$};
  \node[below] at ($1/2*(T3A2) + 1/2*(T3B2)$) {$\vec C_3$};
  
  \foreach \i in {0,...,3}{
    \filldraw (T4A\i) circle (1pt);
    \filldraw (T4B\i) circle (1pt);
    \filldraw (T4C\i) circle (1pt);
    \filldraw (T4D\i) circle (1pt);

    \draw[thick, arrows={-latex}] (T4A\i) -- (T4B\i);
    \draw[thick, arrows={-latex}] (T4B\i) -- (T4C\i);
  }
  \draw[thick, arrows={-latex}] (T4A0) -- (T4C0);
  \draw[thick, arrows={latex-}] (T4A1) -- (T4C1);
  \draw[thick, arrows={latex-}] (T4A2) -- (T4C2);
  \draw[thick, arrows={latex-}] (T4A3) -- (T4C3);

  \draw[thick, arrows={-latex}] (T4A0) -- (T4D0);
  \draw[thick, arrows={-latex}] (T4B0) -- (T4D0);
  \draw[thick, arrows={-latex}] (T4C0) -- (T4D0);

  \draw[thick, arrows={latex-}] (T4A1) -- (T4D1);
  \draw[thick, arrows={latex-}] (T4B1) -- (T4D1);
  \draw[thick, arrows={latex-}] (T4C1) -- (T4D1);

  \draw[thick, arrows={-latex}] (T4A2) -- (T4D2);
  \draw[thick, arrows={-latex}] (T4B2) -- (T4D2);
  \draw[thick, arrows={-latex}] (T4C2) -- (T4D2);

  \draw[thick, arrows={-latex}] (T4A3) -- (T4D3);
  \draw[thick, arrows={-latex}] (T4B3) -- (T4D3);
  \draw[thick, arrows={latex-}] (T4C3) -- (T4D3);

  \node[below] at ($1/2*(T4A0) + 1/2*(T4B0)$) {$\Tr_4$};
  \node[below] at ($1/2*(T4A1) + 1/2*(T4B1)$) {$W_4$};
  \node[below] at ($1/2*(T4A2) + 1/2*(T4B2)$) {$L_4$};
  \node[below] at ($1/2*(T4A3) + 1/2*(T4B3)$) {$R_4$};

  \filldraw (TpB1) circle (1pt);
  \filldraw (TpA2) circle (1pt);
  \filldraw (TpC2) circle (1pt);

  \draw[thick, arrows={-latex}] (TpA2) -- (TpC2);

  \node[left] at (TpB1) {$1$};
  \node[left] at (TpA2) {$1$};
  \node[right] at (TpC2) {$2$};

  \node[below] at (TpB1) {$1$};
  \node[below] at (TpB2) {$A$};

  \filldraw (F2A1) circle (1pt);
  \filldraw (F2B1) circle (1pt);
  \filldraw (F2A2) circle (1pt);
  \filldraw (F2B2) circle (1pt);  

  \draw[thick, arrows={-latex}] (F2A1) -- (F2B1);
  \draw[thick, arrows={latex-}] (F2A2) -- (F2B2);

  \node[left] at (F2A1) {$1$};
  \node[left] at (F2A2) {$1$};

  \node[below] at ($1/2*(F2A1) + 1/2*(F2B1)$) {$\alpha$};
  \node[below] at ($1/2*(F2A2) + 1/2*(F2B2)$) {$\beta$};

  \foreach \i in {0,...,3}{
    \filldraw (F3A\i) circle (1pt);
    \filldraw (F3B\i) circle (1pt);
    \filldraw (F3C\i) circle (1pt);

    \draw[thick, arrows={-latex}] (F3A\i) -- (F3B\i);

    \node[left] at (F3A\i) {$1$};
    \node[right] at (F3B\i) {$2$};
  }
  \draw[thick, arrows={-latex}] (F3A0) -- (F3C0);
  \draw[thick, arrows={-latex}] (F3B0) -- (F3C0);

  \draw[thick, arrows={latex-}] (F3A1) -- (F3C1);
  \draw[thick, arrows={latex-}] (F3B1) -- (F3C1);

  \draw[thick, arrows={-latex}] (F3A2) -- (F3C2);
  \draw[thick, arrows={latex-}] (F3B2) -- (F3C2);

  \draw[thick, arrows={latex-}] (F3A3) -- (F3C3);
  \draw[thick, arrows={-latex}] (F3B3) -- (F3C3);

  \node[below] at ($1/2*(F3A0) + 1/2*(F3B0)$) {$O^A$};
  \node[below] at ($1/2*(F3A1) + 1/2*(F3B1)$) {$I^A$};
  \node[below] at ($1/2*(F3A2) + 1/2*(F3B2)$) {$\Tr_3^A$};
  \node[below] at ($1/2*(F3A3) + 1/2*(F3B3)$) {$\vec C_3^A$};
\end{tikzpicture}
    \caption{Types and flags used.}
    \label{fig:typesandflags}
  \end{center}
\end{figure}

We have already introduced the notation~$\Tr_k$ to denote the transitive tournament
of order~$k$ and the notation for all the other tournaments of orders~$3$ and~$4$,
but we repeat them below for the readers convenience.
\begin{itemize}
\item The tournament~$\vec C_3$ is the~$3$-directed cycle;
\item The tournament~$R_4$ is the (unique) tournament of order~$4$ that has outdegree
  sequence~$(1,1,2,2)$;
\item The tournament~$W_4$ is the (unique) \emph{non-transitive} tournament of
  order~$4$ that has a vertex with outdegree~$3$ (that is, there is a ``winner''
  in~$W_4$);
\item The tournament~$L_4$ is the (unique) \emph{non-transitive} tournament of
  order~$4$ that has a vertex with indegree~$3$ (that is, there is a ``loser''
  in~$L_4$).
\end{itemize}

We define the~$1$-flag~$\alpha$ as the (unique)~$1$-flag of order~$2$ in which the
labelled vertex beats the unlabelled vertex and~$\beta$ as the other~$1$-flag of
order~$2$. We also define the following~$A$-flags of order~$3$.
\begin{itemize}
\item The flag~$O^A$, in which the only unlabelled vertex is beaten by both labelled
  vertices;
\item The flag~$I^A$, in which the only unlabelled vertex beats both labelled
  vertices;
\item The flag~$\Tr_3^A$, which is the only remaining~$A$-flag whose underlying model
  is~$\Tr_3$;
\item The flag~$\vec C_3^A$, which is the only~$A$-flag whose underlying model
  is~$\vec C_3$.
\end{itemize}
This is the complete list of~$A$-flags of order~$3$.

We also follow the original notation of flag algebras when using the downward
operator~$\llb{}\cdot{}\rrb_\sigma$ to the~$0$-algebra or when
using~$\sigma$-extensions of homomorphisms~$\phi\in\Hom^+(\cA^0,\RR)$ (which are
denoted by~$\bm{\phi^\sigma}$). We remind that~$\bm{\phi^\sigma}$ can be conveniently
viewed~\cite[Definition~10]{R:FlagAlgebras} as the
unique~$\Hom^+(\cA^\sigma,\RR)$-valued random variables satisfying the identities
\begin{align}\label{extensions}
\expect{\bm{\phi^\sigma}(F)} & =
\frac{\phi(\llb F\rrb_\sigma)}{\phi(\llbracket\One_\sigma\rrbracket_\sigma)}
\end{align}
for every~$F\in\cF^\sigma$.

Finally, we recall a very useful way to obtain the probability measure
of~$\bm{\phi^\sigma}$.

If~$F$ is a~$0$-flag and~$\sigma$ is a type such that~$p(\sigma,F)>0$ (when
regarding~$\sigma$ as a~$0$-flag), then we consider the following random
experiment. Choose uniformly at random an embedding~$\bm{\theta}$ of~$\sigma$ in~$F$
and for every Borel subset~$A$ of~$[0,1]^{\cF^\sigma}$, define
(see~\cite[Definition~9]{R:FlagAlgebras})
\begin{align*}
  \PP^\sigma_F(A) & = \PP(p^{(F,\bm{\theta})} \in A),
\end{align*}
where~$p^F$ denotes the linear functional~$p({{}\cdot{}},F)$, which can be regarded
as a point of~$[0,1]^{\cF^\sigma}$.

Recall~\cite[Theorem~3.12]{R:FlagAlgebras} that if~$(F_n)_{n\in\NN}$ is a convergent
sequence converging to~$\phi$, then the sequence of probability
measures~$(\PP^\sigma_{F_n})_{n\in\NN}$ on Borel subsets of~$[0,1]^{\cF^\sigma}$
weakly converges to the probability measure~$\PP^\sigma$ of~$\bm{\phi^\sigma}$.

We will not need these concepts in the more complicated scenario when the smaller
type is also non-trivial.

In this note, the most useful property of weak convergence of probability measures is
the following.

\begin{proposition}\label{prop:weakconv}
  If~$X$ is a metrizable space, $\PP$ is a Borel probability measure on~$X$
  and~$(\PP_n)_{n\in\NN}$ is a sequence of Borel probability measures on~$X$, then
  the following are equivalent.
  \begin{itemize}
  \item The sequence~$(\PP_n)_{n\in\NN}$ weakly converges to~$\PP$;
  \item For every~$A\subset X$ with~$\PP(\delta A)=0$ (where~$\delta A$ is the
    boundary of~$A$), we have
    \begin{align*}
      \lim_{n\to\infty}\PP_n(A) & = \PP(A);
    \end{align*}
  \item For every~$A\subset X$ open, we have
    \begin{align*}
      \liminf_{n\to\infty}\PP_n(A) & \geq \PP(A).
    \end{align*}
  \end{itemize}
\end{proposition}

We have already introduced the notation~$\phiqr$ to denote the homomorphism
of~$\Hom^+(\cA^0,\RR)$ corresponding to the random tournament, that is, it is the
almost sure limit of the sequence of random tournaments~$(\bm{T_{1/2}(n)})_{n\in\NN}$
(where each arc orientation is picked independently at random with probability~$1/2$)
when the number of vertices goes to infinity.

\smallskip

As we said in the introduction, the~$Q$ properties of
Chung--Graham~\cite{CG:QuasiRandomTournaments} of a sequence of
tournaments~$(T_n)_{n\in\NN}$ with~$\lv V(T_n)\rv = n$ that we are interested in are
the following.
\begin{itemize}
\item $Q_1$: $\lim_{n\to\infty}p(\Tr_3,T_n) = 3/4$ and~$\lim_{n\to\infty}p(\vec
  C_3,T_n) = 1/4$;
\item $Q_2$: $p(\vec C_3,T_n)$ is asymptotically maximized by the
  sequence~$(T_n)_{n\in\NN}$;
\item $Q_3$: The sequence of tournaments~$(T_n)_{n\in\NN}$ of increasing orders is
  \emph{almost balanced}, that is, all but~$o(n)$ vertices of~$T_n$ have
  outdegree~$(1/2 + o(1))n$.
\end{itemize}

If we assume that this sequence converges to a
homomorphism~$\phi\in\Hom^+(\cA^0,\RR)$, then these properties are translated to the
following properties of~$\phi$.
\begin{itemize}
\item $Q_1$: $\phi(\Tr_3) = 3/4$ and~$\phi(\vec C_3) = 1/4$;
\item $Q_2$: $\phi(\vec C_3)$ is maximum, i.e., we have~$\phi(\vec C_3) =
  \max\{\psi(\vec C_ 3) : \psi\in\Hom^+(\cA^0,\RR)\}$;
\item $Q_3$: $\bm{\phi^1}(\alpha) = 1/2$ a.s.
\end{itemize}

Note that since~$\vec C_3+\Tr_3 = \One_0$, it is enough to check only one of the
values in~$Q_1$. Furthermore, since~$\alpha + \beta = \One_1$, we immediately get
that~$Q_3$ is equivalent to~$\bm{\phi^1}(\beta) = 1/2$ a.s.\ and equivalent
to~$\bm{\phi^1}(\alpha) = \bm{\phi^1}(\beta)$ a.s.

We call a homomorphism~$\phi\in\Hom^+(\cA^0,\RR)$ \emph{balanced} if it
satisfies any (and therefore all) of these properties.

We now prove a small lemma that adds one other item to this list of properties.

\begin{lemma}\label{lem:Tr4R4}
  In the theory of tournaments, if~$\phi\in\Hom^+(\cA^0,\RR)$,
  then~$\phi(\Tr_4)\geq\phi(R_4)$ with equality if and only~$\phi$ is balanced.
\end{lemma}

\begin{proof}
  It is easy to check the following flag~algebra identity.
  \begin{align*}
    \vec C_3 & = \frac{1}{4} + \frac{1}{4}R_4 - \frac{1}{4}\Tr_4.
  \end{align*}

  From~$Q_1$ and~$Q_2$, we know that~$\phi(\vec C_3)\leq 1/4$, with equality if and
  only if~$\phi$ is balanced; this directly implies that~$\phi(\Tr_4)\geq \phi(R_4)$,
  with equality if and only if~$\phi$ is balanced.
\end{proof}

\section{The Carousel Homomorphism}
\label{sec:phiR}

Stemming from Proposition~\ref{prop:loctranW4L4}, let us call a
homomorphism~$\phi\in\Hom^+(\cA^0,\RR)$ \emph{locally transitive} if we
have~$\phi(W_4+L_4)=0$.

Note that the fact that a sequence of tournaments~$(T_n)_{n\in\NN}$ converges to a
locally transitive homomorphism does \emph{not} imply that the tournaments are
locally transitive. Rather, it only implies that the density of~$W_4$ and~$L_4$ go to
zero as~$n$ goes to infinity, that is, the sequence is only \emph{asymptotically
  locally transitive}.

However, every locally transitive homomorphism~$\phi$ is also an algebra homomorphism
in the theory of locally transitive tournaments (i.e., the theory of tournaments that
avoid both~$W_4$ and~$L_4$), hence there exists a sequence of locally transitive
tournaments converging to~$\phi$.

Now we claim that the sequence of carousel tournaments~$(R_{2n+1})_{n\in\NN}$
is convergent, but we defer the proof of this claim to Section~\ref{sec:conv}. We
will call the limit of this sequence the \emph{carousel homomorphism} and we
will denote it by~$\phi_R$.

We now list a series of properties of a homomorphism~$\phi\in\Hom^+(\cA^0,\RR)$ that
we will prove to hold if and only if~$\phi=\phi_R$. Property~$S_1$ is stated just for
practical reasons and the equivalence of properties~$S_1$ and~$S_2$ implies
that~$\phi_R$ is the \emph{only} homomorphism that is both balanced and locally
transitive.

\begin{itemize}
\item $S_1$: $\phi = \phi_R$;
\item $S_2$: $\phi$ is balanced and locally transitive;
\item $S_3$: $\phi$ maximizes the density of~$R_4$, i.e., we have
  \begin{align*}
    \phi(R_4) & = \max\{\psi(R_4) : \psi\in\Hom^+(\cA^0,\RR)\};
  \end{align*}
\item $S_4$: $\phi$ maximizes the second moment of~$\bm{\phi^A}(\vec C_3^A)$.
\end{itemize}

For the next properties, it will be more practical to state them with free
parameters~$F$ and~$q$, which will be respectively an~$A$-algebra element and a real
number (not any element and real number!).
\begin{itemize}
\item $S_5(F,q)$: $\bm{\phi^A}(F)\sim U(0,q)$ (that is, the random
  variable~$\bm{\phi^A}(F)$ is uniformly distributed in~$[0,q]$);
\item $S_6(F,q)$: $\phi$ maximizes the second moment of~$\bm{\phi^A}(F)$ restricted
  to~$\expect{\bm{\phi^A}(F)}=q$, i.e., we have~$\expect{\bm{\phi^A}(F)}=q$ and
  \begin{align*}
    \expect{\bm{\phi^A}(F)^2} & =
    \max\{\expect{\bm{\psi^A}(F)^2} :
    \psi\in\Hom^+(\cA^0,\RR) \text{ with } \expect{\bm{\psi^A}(F)}=q\}.
  \end{align*}
\end{itemize}

We can now state the theorem.

\begin{theorem}\label{thm:charac}
  If~$F$ is an~$A$-flag of order~$3$ and~$G$ is either~$O^A+I^A$ or~$\vec C_3^A +
  \Tr_3^A$, then
  \begin{align*}
    S_1 & \Rightarrow S_2 \Rightarrow S_3 \Rightarrow S_4 
    \Rightarrow S_5(F,1/2) \Rightarrow S_6(F,1/4)
    \Rightarrow S_5(G,1) \Rightarrow S_6(G,1/2) \Rightarrow S_1.
  \end{align*}
\end{theorem}

We will establish Theorem~\ref{thm:charac} through a series of lemmas, enlarging the
family of properties known to be equivalent after each lemma.

\begin{lemma}\label{lem:S12}
  We have~$S_1\iff S_2$.
\end{lemma}

\begin{proof}
  Since~$R_{2n+1}$ is both balanced and locally transitive for every~$n\in\NN$, it
  follows that~$\phi_R$ is balanced and locally transitive.

  Suppose that~$\phi\in\Hom^+(\cA^0[T_{\text{Tournaments}}],\RR)$
  satisfies~$S_2$ and let~$T_{\neg\{W_4,L_4\}}$ be the theory of tournaments without
  any occurrence of~$W_4$ or~$L_4$ (i.e., the theory of locally transitive
  tournaments). Note that~$\phi$ is also an element
  of~$\Hom^+(\cA^0[T_{\neg\{W_4,L_4\}}],\RR)$, hence there exists a
  sequence~$(T_n)_{n\in\NN}$ of tournaments in~$T_{\neg\{W_4,L_4\}}$ converging
  to~$\phi$ and we can take this sequence to be such that~$\lv V(T_n)\rv$ is odd for
  every~$n\in\NN$.

  Since~$\phi$ is balanced, we know that all but~$o(\lv V(T_n)\rv)$ vertices of~$T_n$
  have outdegree~$(1/2 + o(1))\lv V(T_n)\rv$ hence, considering the cyclic ordering
  of~$T_n$ given by Proposition~\ref{prop:loctranchar}, we see that we can
  obtain~$R_{\lv V(T_n)\rv}$ from~$T_n$ by flipping~$o(\lv V(T_n)\rv^2)$ arcs
  of~$T_n$. Since this flipping operation does not change the limit homomorphism, we
  have that~$(T_n)_{n\in\NN}$ converges to the same limit as a subsequence
  of~$(R_{2n+1})_{n\in\NN}$. Therefore, we have~$\phi=\phi_R$.
\end{proof}

\begin{lemma}\label{lem:S13}
  We have~$S_1 \iff S_3$.
\end{lemma}

\begin{proof}
  Let us prove first that~$\phi_R$ satisfies~$S_3$.

  Note that Lemma~\ref{lem:Tr4R4} immediately gives that~$\phi(R_4)\leq 1/2$
  for every~$\phi\in\Hom^+(\cA^0,\RR)$.

  Since~$S_1\iff S_2$ by Lemma~\ref{lem:S12}, we have that~$\phi_R$ is balanced,
  hence Lemma~\ref{lem:Tr4R4} gives~$\phi_R(\Tr_4)=\phi_R(R_4)$. But also, we
  have~$\phi_R(W_4+L_4)=0$ by~$S_2$, hence~$\phi(\Tr_4+R_4)=1$, which
  implies~$\phi_R(R_4)=1/2$.

  Therefore~$S_1\implies S_3$.

  \medskip

  Suppose now that~$\phi\in\Hom^+(\cA^0,\RR)$ maximizes~$\phi(R_4)$. Then we must
  have~$\phi(R_4)=1/2$. On the other hand, since~$\phi(\Tr_4+R_4)\leq 1$, a double
  application of Lemma~\ref{lem:Tr4R4} implies that~$\phi(\Tr_4)=1/2$ and that~$\phi$
  is balanced, hence~$\phi$ satisfies~$S_2$ (since~$\phi(W_4+L_4) = 1 -
  \phi(\Tr_4+R_4)$).

  Therefore~$S_3\implies S_1$ (by Lemma~\ref{lem:S12}).
\end{proof}

Note that the proof of Lemma~\ref{lem:S13} also established the following corollary.
\begin{corollary}\label{cor:maxR4}
  In the theory of tournaments, if~$\phi\in\Hom^+(\cA^0,\RR)$, then~$\phi(R_4)\leq
  1/2$, with equality if and only if~$\phi=\phi_R$.
\end{corollary}

Let us continue with the proof of Theorem~\ref{thm:charac}.

\begin{lemma}\label{lem:S14}
  We have~$S_1\iff S_4$.
\end{lemma}

\begin{proof}
  Note that
  \begin{align*}
    \expect{\bm{\phi^A}(\vec C_3^A)^2} & = \frac{\phi(R_4)}{6},
  \end{align*}
  hence~$\phi$ maximizes the second moment of~$\bm{\phi^A}(\vec C_3^A)$ if and only
  if~$\phi$ maximizes the density of~$R_4$, so the result follows from
  Corollary~\ref{cor:maxR4}.
\end{proof}

\begin{lemma}\label{lem:S156F}
  If~$F$ is an~$A$-flag of order~$3$, then~$S_1\iff S_5(F,1/2) \iff S_6(F,1/4)$.
\end{lemma}

\begin{proof}
  Let us first prove that~$S_1$ implies~$S_5(F,1/2)$.

  Let~$\PP^A$ be the Borel probability measure of~$\bm{\phi_R^A}$ and for
  every~$a\leq b$, let
  \begin{align*}
    B_{a,b}(F) & = \{x\in[0,1]^{\cF^A} : a < x_F < b\}.
  \end{align*}

  Note that~$B_{a,b}$ is an open subset
  of~$[0,1]^{\cF^A}$. Since~$(\PP^A_{R_{2n+1}})_{n\in\NN}$ weakly converges
  to~$\PP^A$, by Proposition~\ref{prop:weakconv}, it is enough to prove that
  \begin{align*}
    \liminf_{n\to\infty}\PP^A_{R_{2n+1}}(B_{a,b}(F)) & = 2(b-a),
  \end{align*}
  for every~$0\leq a\leq b\leq 1/2$; and
  \begin{align*}
    \liminf_{n\to\infty}\PP^A_{R_{2n+1}}(B_{a,b}(F)) & = 1-2a,
  \end{align*}
  for every~$0\leq a\leq 1/2\leq b\leq 1$.

  Recall the definition of~$\PP^A_{R_{2n+1}}$: consider the random experiment where
  we pick at random an embedding~$\bm{\theta}$ of~$A$ in~$R_{2n+1}$, then we have
  \begin{align*}
    \PP^A_{R_{2n+1}}(B_{a,b}(F)) & = \PP(a < p(F,\bm{L_{2n+1}}) < b),
  \end{align*}
  where~$\bm{L_{2n+1}}$ is the random~$A$-flag~$(R_{2n+1},\bm{\theta})$.

  Note that since~$\bm{\theta}$ is an embedding of~$A$ in~$R_{2n+1}$, we must have
  \begin{align*}
    \bm{\theta}(2) & = (\bm{\theta}(1) + \bm{i}) \bmod (2n+1),
  \end{align*}
  for some (random)~$\bm{i}\in[n]$. Note also that from the symmetry of~$R_{2n+1}$,
  the variable~$\bm{i}$ has uniform distribution in~$[n]$.

  Let~$j\in[2n]$
  and~$\bm{J}=\{\bm{\theta}(1),\bm{\theta}(2),(\bm{\theta}(1)+j)\bmod(2n+1)\}$. Note that
  we have the following (see Figure~\ref{fig:Aspan}).
  \begin{itemize}
  \item If~$j < \bm{i}$, then~$\bm{J}$ induces an occurrence of~$\Tr_3^A$;
  \item If~$\bm{i} < j \leq n$, then~$\bm{J}$ induces an occurrence of~$O^A$;
  \item If~$n < j \leq \bm{i} + n$, then~$\bm{J}$ induces an occurrence of~$\vec C_3^A$;
  \item If~$\bm{i} + n < j$, then~$\bm{J}$ induces an occurrence of~$I^A$.
  \end{itemize}

  \begin{figure}[ht]
    \begin{center}
      \begin{tikzpicture}[scale=0.9]
\begingroup
\small
\def\shorten{0.2cm}
\def\toneangle{90}
\def\firststartangle{75}
\def\firstangle{40}
\def\firstendangle{5}
\def\ttwoangle{-10}
\def\secondstartangle{-25}
\def\secondangle{-55}
\def\tonepnangle{-70}
\def\secondendangle{-85}
\def\thirdstartangle{-95}
\def\thirdangle{-140}
\def\ttwopnangle{-170}
\def\thirdendangle{-185}
\def\fourthstartangle{-195}
\def\fourthangle{-225}
\def\fourthendangle{-255}

\def\inradius{2.2cm}
\def\pointradius{2.7cm}
\def\outradius{3.2cm}
\def\labelradius{3.9cm}

\coordinate (T1)   at (\toneangle:\inradius);
\coordinate (T2) at (\ttwoangle:\inradius);
\coordinate (T1PN)   at (\tonepnangle:\pointradius);
\coordinate (T2PN)   at (\ttwopnangle:\pointradius);

\coordinate (LT1)   at (\toneangle:\inradius);
\coordinate (LT2) at (\ttwoangle:\inradius);
\coordinate (LT1PN)   at (\tonepnangle:\labelradius);
\coordinate (LT2PN)   at (\ttwopnangle:\outradius);

\coordinate (FIRST) at (\firstangle:\inradius);
\coordinate (SECOND) at (\secondangle:\inradius);
\coordinate (THIRD) at (\thirdangle:\inradius);
\coordinate (FOURTH) at (\fourthangle:\inradius);

\foreach \p in {T1,T2,T1PN,T2PN}
\filldraw (\p) circle (2pt);

\foreach \a in {%
  \firststartangle,\firstendangle,
  \secondstartangle,\secondendangle,
  \thirdstartangle,\thirdendangle,
  \fourthstartangle,\fourthendangle%
}
\draw (\a:\inradius) -- (\a:\outradius);

\foreach \s/\e in {%
  \firststartangle/\firstendangle,
  \secondstartangle/\secondendangle,
  \thirdstartangle/\thirdendangle,
  \fourthstartangle/\fourthendangle%
}{
  \foreach \r in {\outradius,\inradius}
  \draw (\s:\r) arc (\s:\e:\r);
}

\draw[shorten >=\shorten,shorten <=\shorten, thick, arrows={-latex}] (T1) -- (FIRST);
\draw[shorten >=\shorten,shorten <=\shorten, thick, arrows={-latex}] (T1) -- (SECOND);
\draw[shorten >=\shorten,shorten <=\shorten, thick, arrows={-latex}] (THIRD) -- (T1);
\draw[shorten >=\shorten,shorten <=\shorten, thick, arrows={-latex}] (FOURTH) -- (T1);
\draw[shorten >=\shorten,shorten <=\shorten, thick, arrows={-latex}] (FIRST) -- (T2);
\draw[shorten >=\shorten,shorten <=\shorten, thick, arrows={-latex}] (T2) -- (SECOND);
\draw[shorten >=\shorten,shorten <=\shorten, thick, arrows={-latex}] (T2) -- (THIRD);
\draw[shorten >=\shorten,shorten <=\shorten, thick, arrows={-latex}] (FOURTH) -- (T2);

\draw[shorten >=\shorten,shorten <=\shorten, thick, dashed, arrows={-latex}] (T1) -- (T2);

\node[above] at (LT1) {$\bm{\theta}(1)$};
\node[right] at (LT2) {$\bm{\theta}(2) = (\bm{\theta}(1) + \bm{i}) \bmod (2n+1)$};
\node at (LT1PN) {$(\bm{\theta}(1)+n) \bmod (2n+1)$};
\node[left] at (LT2PN) {$(\bm{\theta}(2)+n) \bmod (2n+1)$};
\endgroup
\end{tikzpicture}
      \caption{Neighbourhoods of~$\bm{\theta}(1)$ and~$\bm{\theta}(2)$.}
      \label{fig:Aspan}
    \end{center}
  \end{figure}
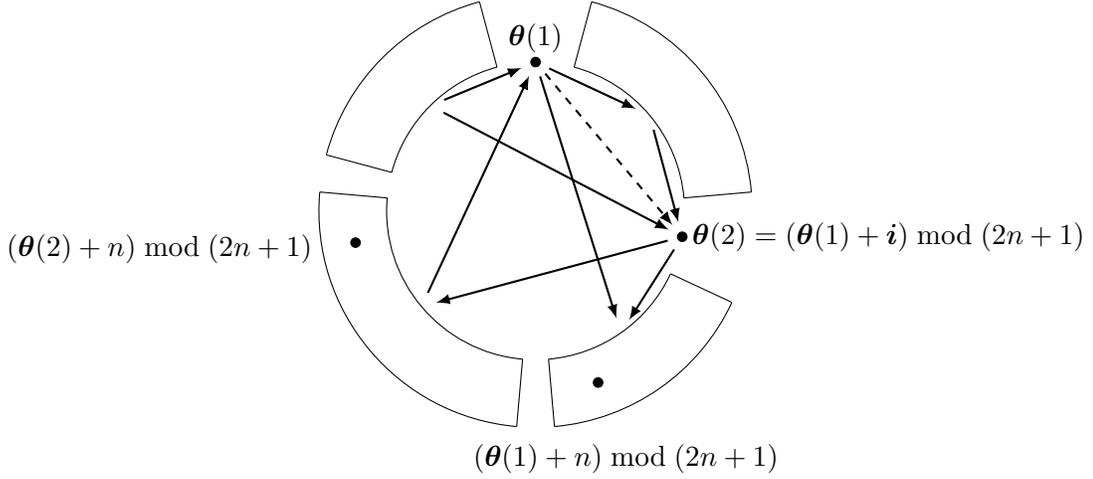

  This implies that
  \begin{align*}
    p(\Tr_3^A,\bm{L_{2n+1}}) & = \frac{\bm{i}-1}{2n-1}; &
    p(O^A,\bm{L_{2n+1}}) & = \frac{n-\bm{i}}{2n-1};\\
    p(\vec C_3^A,\bm{L_{2n+1}}) & = \frac{\bm{i}}{2n-1}; &
    p(I^A,\bm{L_{2n+1}}) & = \frac{n-\bm{i}}{2n-1}.
  \end{align*}

  Hence, since~$\bm{i}$ has uniform distribution over~$[n]$, we get
  that~$p(\Tr_3^A,\bm{L_{2n+1}})$, $p(O^A,\bm{L_{2n+1}})$ and~$p(I^A,\bm{L_{2n+1}})$
  have uniform distribution over~$\{t/(2n-1) : t\in
  \{0,1,\ldots,n-1\}\}$. Moreover~$p(\vec C_3^A,\bm{L_{2n+1}})$ has uniform
  distribution over~$\{t/(2n-1) : t\in[n]\}$.

  Letting~$n\to\infty$, it follows that
  \begin{align*}
    \liminf_{n\to\infty}\PP^A_{R_{2n+1}}(B_{a,b}(F)) & = 2(b-a),
  \end{align*}
  for every~$0\leq a\leq b\leq 1/2$; and
  \begin{align*}
    \liminf_{n\to\infty}\PP^A_{R_{2n+1}}(B_{a,b}(F)) & = 1-2a,
  \end{align*}
  for every~$0\leq a\leq 1/2\leq b\leq 1$ as desired.

  Therefore~$S_1\implies S_5(F,1/2)$.

  \medskip

  Now let us prove that~$S_5(F,1/2)\implies S_6(F,1/4)$.

  Suppose~$\psi$ is such that~$\expect{\bm{\psi^A}(F)}=1/4$.

  If~$F=\vec C_3^A$, then we have
  \begin{align*}
    \frac{1}{4} & = \expect{\bm{\psi^A}(F)} = \psi(\vec C_3),
  \end{align*}
  hence~$\psi$ is balanced.

  If~$F$ is one of~$O^A$, $I^A$ or~$\Tr_3^A$, then we have
  \begin{align*}
    \frac{1}{4} & = \expect{\bm{\psi^A}(F)} = \frac{\psi(\Tr_3)}{3},
  \end{align*}
  which yields~$\psi(\Tr_3)=3/4$, hence~$\psi$ is balanced.

  Therefore every~$\psi$ with~$\expect{\bm{\psi^A}(F)}=1/4$ must be balanced.

  Since the second moment of a~$U(0,1/2)$-random variable is~$1/12$, it is enough to
  prove that if~$\psi$ is balanced, then~$\bm{\psi^A}(F)\leq 1/12$.

  If~$F=\vec C_3^A$, then we have
  \begin{align}\label{eqn:2ndmomC3}
    \expect{\bm{\psi^A}(F)^2} & = \frac{\psi(R_4)}{6} \leq \frac{1}{12},
  \end{align}
  since the maximum value of~$\psi(R_4)$ is~$1/2$ (Corollary~\ref{cor:maxR4}).

  On the other hand, if~$F$ is one of~$O^A$, $I^A$ or~$\Tr_3^A$, then we have
  \begin{align}\label{eqn:2ndmomF}
    \expect{\bm{\psi^A}(F)^2} & = \frac{\psi(\Tr_4)}{6} = \frac{\psi(R_4)}{6}
    \leq \frac{1}{12},
  \end{align}
  by Lemma~\ref{lem:Tr4R4} and Corollary~\ref{cor:maxR4}.

  Therefore~$S_5(F,1/2)\implies S_6(F,1/4)$.

  \medskip

  Finally, let us prove that~$S_6(F,1/4)$ implies~$S_1$.

  If~$\phi$ satisfies~$S_6(F,1/4)$, we have already proved that it must be balanced
  (since~$\expect{\bm{\phi^A}(F)} = 1/4$) and from the equation part
  of~\eqref{eqn:2ndmomC3} and~\eqref{eqn:2ndmomF} and the fact that the second moment
  of a $U(0,1/2)$-random variable is~$1/12$, we have that~$\phi(R_4) \geq 1/2$,
  hence~$\phi=\phi_R$ by Corollary~\ref{cor:maxR4}.
\end{proof}

\begin{lemma}\label{lem:S156G}
  If~$G$ is either~$O^A+I^A$ or~$\vec C_3^A + \Tr_3^A$, then~$S_1\iff S_5(G,1)\iff
  S_6(G,1/2)$.
\end{lemma}

\begin{proof}
  (The proof is somewhat analogous to the proof of Lemma~\ref{lem:S156F}.)

  To prove that~$S_1\implies S_5(G,1)$, repeat the part~$S_1\implies S_5(F,1/2)$ of
  the proof of Lemma~\ref{lem:S156F} and note that since
  \begin{align*}
    p(\Tr_3^A,\bm{L_{2n+1}}) & = \frac{\bm{i}-1}{2n-1}; &
    p(O^A,\bm{L_{2n+1}}) & = \frac{n-\bm{i}}{2n-1};\\
    p(\vec C_3^A,\bm{L_{2n+1}}) & = \frac{\bm{i}}{2n-1}; &
    p(I^A,\bm{L_{2n+1}}) & = \frac{n-\bm{i}}{2n-1};
  \end{align*}
  we have that~$p(O^A+I^A,\bm{L_{2n+1}})$ has uniform distribution on~$\{2t/(2n-1) :
  t\in\{0,1,\ldots,n-1\}\}$ and that~$p(\vec C_3^A + \Tr_3^A,\bm{L_{2n+1}})$ has
  uniform distribution on~$\{(2t-1)/(2n-1) : t\in[n]\}$.

  Letting~$n\to\infty$, it follows that
  \begin{align*}
    \liminf_{n\to\infty}\PP^A_{R_{2n+1}}(a < p(F,\bm{L_{2n+1}}) < b) & = b-a,
  \end{align*}
  for every~$0\leq a\leq b\leq 1$, which implies~$S_5(G,1)$.

  \medskip

  Now let us prove that~$S_5(G,1)\implies S_6(G,1/2)$.

  Suppose~$\psi$ is such that~$\expect{\bm{\psi^A}(G)}=1/2$.

  If~$G=O^A+I^A$, then we have
  \begin{align*}
    \frac{1}{2} & = \expect{\bm{\psi^A}(G)} = \frac{2\psi(\Tr_3)}{3},
  \end{align*}
  which yields~$\psi(\Tr_3)=3/4$, hence~$\psi$ is balanced.

  If~$G=\vec C_3^A + \Tr_3^A$, then we have
  \begin{align*}
    \frac{1}{2} & = \expect{\bm{\psi^A}(G)} = \psi(\vec C_3) + \frac{\psi(\Tr_3)}{3}
    = \frac{1}{3} + \frac{2\psi(\vec C_3)}{3},
  \end{align*}
  which yields~$\psi(\vec C_3)=1/4$, hence~$\psi$ is balanced.

  Therefore every~$\psi$ with~$\expect{\bm{\psi^A}(G)}=1/2$ must be balanced.

  Since the second moment of a~$U(0,1)$-random variable is~$1/3$, it is enough to
  prove that if~$\psi$ is balanced, then~$\bm{\psi^A}(G)\leq 1/3$.

  But note that, if~$G=O^A+I^A$, then we have
  \begin{align}\label{eqn:2ndmomOAIA}
    \expect{\bm{\psi^A}(G)^2} & = \frac{\Tr_4}{2} + \frac{R_4}{6}
    = \frac{2R_4}{3} \leq \frac{1}{3},
  \end{align}
  by Lemma~\ref{lem:Tr4R4} and Corollary~\ref{cor:maxR4}.

  Furthermore, if~$G=\vec C_3^A + \Tr_3^A$, then we have
  \begin{align}\label{eqn:2ndmomC3ATr3A}
    \expect{\bm{\psi^A}(G)^2} & = \frac{\Tr_4}{6} + \frac{R_4}{2}
    = \frac{2R_4}{3} \leq \frac{1}{3},
  \end{align}
  also by Lemma~\ref{lem:Tr4R4} and Corollary~\ref{cor:maxR4}.

  Therefore~$S_5(G,1)\implies S_6(G,1/2)$.

  \medskip

  Finally, let us prove that~$S_6(G,1/2)$ implies~$S_1$.

  If~$\phi$ satisfies~$S_6(G,1/2)$, we have already proved that it must be balanced
  (since~$\expect{\bm{\phi^A}(G)}=1/2$) and from equation part
  of~\eqref{eqn:2ndmomOAIA} and~\eqref{eqn:2ndmomC3ATr3A} and the fact that the
  second moment of a $U(0,1)$-random variable is~$1/3$, we have that~$\phi(R_4) \geq
  1/2$, hence~$\phi=\phi_R$ by Corollary~\ref{cor:maxR4}.
\end{proof}

This finishes the proof of Theorem~\ref{thm:charac}.

\section{Convergence of the Sequence~$(R_{2n+1})_{n\in\NN}$}
\label{sec:conv}

We present now the proof that the sequence of carousel
tournaments~$(R_{2n+1})_{n\in\NN}$ is convergent. The proof can be obtained by
reinterpreting the proof of Lemma~\ref{lem:S12}.

\begin{proposition}
  The sequence~$(R_{2n+1})_{n\in\NN}$ is convergent.
\end{proposition}

\begin{proof}
  From compactness of~$[0,1]^{\cF^0}$, we know that~$(R_{2n+1})_{n\in\NN}$ must have
  a convergent subsequence, so for every infinite set~$I\subset\NN$ of indexes such
  that the subsequence~$(R_{2i+1})_{i\in I}$ converges,
  let~$\phi_I\in\Hom^+(\cA^0,\RR)$ be its limit. For convenience, let~$\cC$ be the
  set of all~$I\subset\NN$ such that~$(R_{2i+1})_{i\in I}$ converges.

  Now we repeat the proof of Lemma~\ref{lem:S12} using an arbitrary~$I\in\cC$.

  For the forward implication~$S_1\implies S_2$, since~$R_{2n+1}$ is both balanced
  and locally transitive, we have that~$\phi_I$ is balanced and locally transitive
  for every~$I\in\cC$.

  The proof of implication~$S_2\implies S_1$ proceeds a little bit differently: we
  pick the sequence~$(T_n)_{n\in\NN}$ of locally transitive tournaments converging
  to~$\phi$ to be such that
  \begin{align*}
    \{\lv V(T_n)\rv : n\in\NN\} & \subset \{2i+1 : i\in I\}.
  \end{align*}

  To see that this can be done, recall~\cite[Theorem~3.3b]{R:FlagAlgebras} that if we
  define the probability measure~$\PP_n$ over~$\cF^0_n$ as~$\PP_n(F) = \phi(F)$ and
  we pick independently at random for every~$n\in\NN$ the~$0$-flag~$\bm{F_n}$
  according to the measure~$\PP_{f(n)}$, where~$f(n) = \Omega(n^2)$, then the
  sequence~$(\bm{F_n})_{n\in\NN}$ converges almost surely to~$\phi$. Since~$I$
  infinite, we can certainly pick~$f$ such that both~$f(n)=\Omega(n^2)$
  and~$f(\NN)\subset \{2i+1 : i\in I\}$ hold. Thus almost every sample
  of~$(\bm{F_n})_{n\in\NN}$ is a desired sequence~$(T_n)_{n\in\NN}$.

  Again, since~$\phi$ is balanced, we know that we can obtain~$R_{\lv V(T_n)\rv}$
  from~$T_n$ by flipping~$o(\lv V(T_n)\rv^2)$ arcs of~$T_n$ and since this flipping
  operation does not change the limit homomorphism, we have that the
  sequence~$(T_n)_{n\in\NN}$ coverges to the same limit as a subsequence
  of~$(R_{2i+1})_{i\in I}$, hence~$\phi=\phi_I$.

  But this means that, if~$I,J\in\cC$, then, we have
  \begin{align*}
    S_2(\phi_J) \implies \phi_J = \phi_I,
  \end{align*}
  hence every convergent subsequence of~$(R_{2n+1})_{n\in\NN}$ converges to the same
  homomorphism, therefore it must be a convergent sequence from compactness
  of~$[0,1]^{\cF^0}$.
\end{proof}

We remark that the convergence of~$(R_{2n+1})_{n\in\NN}$ can also be proved directly
and that a limit of this sequence in the theory of digraphons
(see~\cite[Section~9]{DJ:GraphLimitsAndExchangeableRandomGraphs}) can be constructed
as follows.

\begin{proposition}\label{prop:digraphon}
  Using the quintuple definition of digraphons, let~$W_{00},W_{11}\:[0,1]^2\to[0,1]$
  be the identically zero functions on~$[0,1]^2$ and~$w\:[0,1]\to[0,1]$ be the
  identically zero function on~$[0,1]$. Furthermore, define the
  functions~$W_{01},W_{10}\:[0,1]^2\to[0,1]$ as follows (see
  Figure~\ref{fig:digraphon}).
  \begin{align*}
    W_{01}(y,x) & = W_{10}(x,y) =
    \begin{dcases*}
      1, & if~$(x-y)\bmod 1 < 1/2$;\\
      0, & if~$(x-y)\bmod 1 \geq 1/2$.
    \end{dcases*}
  \end{align*}

  Under these definitions, the sequence~$(R_{2n+1})_{n\in\NN}$ converges to the
  digraphon~$(W_{00},W_{01},W_{10},W_{11},w)\in\cW_5$, that is, for every
  tournament~$T$ with~$V(T)=[k]$, we have
  \begin{align*}
    \lim_{n\to\infty}p(T,T_n) & = \frac{k!}{\lv\Aut(T)\rv}
    \int_{[0,1]^k}\prod_{ij\in A(T)}W_{10}(x_i,x_j)dx_1dx_2\cdots dx_k,
  \end{align*}
  where~$\Aut(T)$ denotes the group of automorphisms of~$T$.
\end{proposition}

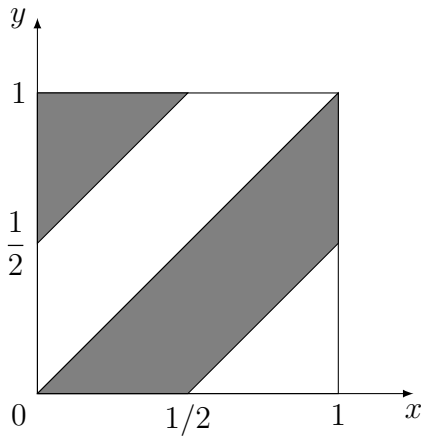
\begin{figure}[ht]
  \begin{center}
    \begin{tikzpicture}[scale=2]
  \foreach \i in {0,1,2}
  \foreach \j in {0,1,2}
  \coordinate (P\i\j) at (\i cm,\j cm);

  \coordinate (P03) at (0cm, 2.5cm);
  \coordinate (P30) at (2.5cm, 0cm);

  \draw (P00) -- (P02) -- (P22) -- (P20) -- cycle;
  \draw (P00) -- (P22);
  \filldraw[fill=gray] (P00) -- (P10) -- (P21) -- (P22) -- cycle;
  \filldraw[fill=gray] (P01) -- (P12) -- (P02) -- cycle;

  \draw[arrows={-latex}] (P00) -- (P03);
  \draw[arrows={-latex}] (P00) -- (P30);

  \node[below] at (P30) {$x$};
  \node[left] at (P03) {$y$};
  \node[below] at (P10) {$1/2$};
  \node[below] at (P20) {$1$};
  \node[left] at (P01) {$\displaystyle \frac{1}{2}$};
  \node[left] at (P02) {$1$};
  \node[below left] at (P00) {$0$};
\end{tikzpicture}
    \caption{The function~$W_{10}$ of Proposition~\ref{prop:digraphon}. The gray area
      represents where the function has value~$1$, the white area represents where
      the function has value~$0$.}
    \label{fig:digraphon}
  \end{center}
\end{figure}

\begin{remark*}
  The factor~$k!/\lv\Aut(T)\rv$ comes from the fact that~$p$ measures unlabelled
  subtournament density and the integral on the right-hand side measures labelled
  subtournament density.
\end{remark*}

\section{Concluding Remarks and Open Problems}
\label{sec:conc}

As we mentioned in the introduction, the problem of minimizing~$\phi(T)$ for a fixed
tournament~$T$ is completely closed but the analogous maximization problem is still
open for very small tournaments. Corollary~\ref{cor:maxR4} completely solves the
maximization of~$\phi(R_4)$, this leaves only one case of order~$4$ still open since
maximizing~$\phi(W_4)$ is analogous to maximizing~$\phi(L_4)$ by flipping all arcs.

For the particular problem of maximizing~$\phi(W_4)$, consider the following
construction (see Figure~\ref{fig:maxW4}). Let~$N$ be an arbitrarily large integer
and~$t\in(0,1)$. Define recursively the sequence~$A_0,A_1,\ldots$ by taking~$A_0 =
[N]$ and by letting~$A_i$ be a subset of~$A_{i-1}$ with size~$t\lv A_i\rv$ (rounded
to the nearest integer) for every~$i>0$. Define the random tournament~$\bm{S_{N,t}}$
through the following procedure: let~$V(\bm{S_{N,t}})=[N]=A_0$, for every~$i > 0$,
every~$v\in A_i$ and every~$w\in A_{i-1}\setminus A_i$, let~$(v,w)\in
A(\bm{S_{N,t}})$ and pick all the remaining arc orientations independently at random
with probability~$1/2$. That is, for every~$i>0$, if~$k = \lv A_{i-1}\setminus
A_i\rv$, then the set~$A_{i-1}\setminus A_i$ spans~$\bm{T_{1/2}(k)}$.

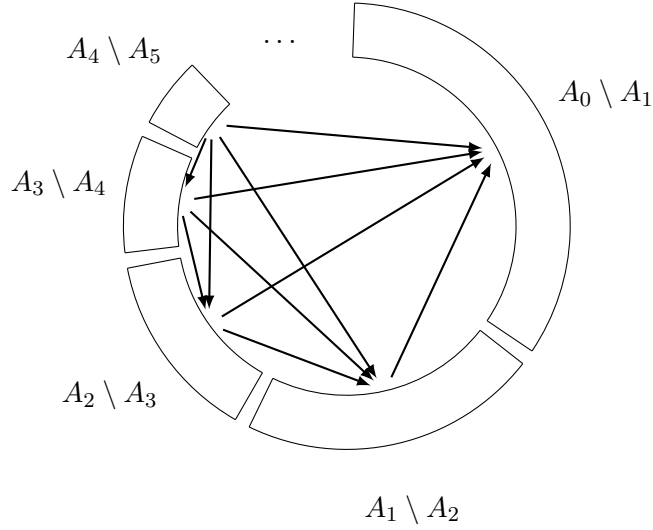
\begin{figure}[ht]
  \begin{center}
    \def\figproportion{0.65}

\begin{tikzpicture}[scale=0.9]
\begingroup
\small
\def\shorten{0.2cm}
\def\halfgap{2}

\def\iterationsmo{4}

\def\inradius{2.5cm}
\def\outradius{3.3cm}
\def\labelradius{4.3cm}

\def\dotsangle{110}
\def\dotsradius{2.9cm}

\foreach[%
  remember=\angle as \prevangle (initially 90),
  remember=\lengthleft as \prevlengthleft (initially 360),
  evaluate=\lengthleft using \figproportion * \prevlengthleft,
  evaluate=\angle using -270+\lengthleft%
] \i in {0,...,\iterationsmo}{
  \pgfmathsetmacro{\startangle}{\prevangle-\halfgap}
  \pgfmathsetmacro{\endangle}{\angle+\halfgap}
  \pgfmathsetmacro{\midangle}{(\startangle + \endangle) / 2}

  \draw (\startangle:\outradius) arc (\startangle:\endangle:\outradius);
  \draw (\startangle:\inradius) arc (\startangle:\endangle:\inradius);

  \draw (\startangle:\inradius) -- (\startangle:\outradius);
  \draw (\endangle:\inradius) -- (\endangle:\outradius);

  \coordinate (P\i) at (\midangle:\inradius);
  \coordinate (L\i) at (\midangle:\labelradius);
}

\foreach \i in {0,...,\iterationsmo}{
  \pgfmathtruncatemacro{\nexti}{\i+1}
  \node at (L\i) {$A_{\i}\setminus A_{\nexti}$};
}

\foreach[remember=\i as \pi (initially 0)] \i in {1,...,\iterationsmo}
\foreach \j in {0,...,\pi}
\draw[shorten >=\shorten,shorten <=\shorten, thick, arrows={-latex}] (P\i) -- (P\j);

\node at (\dotsangle:\dotsradius) {$\ldots$};

\endgroup
\end{tikzpicture}
    \caption{Typical structure of the random tournament~$\bm{S_{N,t}}$. The arcs in
      the picture represent arcs between vertices in distinct parts~$A_{i-1}\setminus
      A_i$. The arcs completely contained any part~$A_{i-1}\setminus A_i$ have their
      orientation picked independently at random with probability~$1/2$ for each
      orientation. This figure uses~$t = \figproportion$, which makes it easier to
      see the structure of the construction but is far from the value of~$t$ that
      maximizes~$\phi_t(W_4)$.}
    \label{fig:maxW4}
  \end{center}
\end{figure}

It is (somewhat) easy to see that~$(\bm{S_{N,t}})_{N\in\NN}$ converges almost
surely to a limit homomorphism~$\phi_t$ such that
\begin{align*}
  \phi_t(W_4) & = \left.(1-t)^3\(t + \frac{1-t}{8}\)\middle\slash(1-t^4)\right..
\end{align*}

Certainly, every value of~$\phi_t(W_4)$ for~$t\in(0,1)$ is a lower bound for the
maximization problem for~$W_4$. The maximum of~$\phi_t(W_4)$ (which can be computed with
standard calculus arguments) is
\begin{align*}
  \max\{\phi_t(W_4) : t\in(0,1)\} & = 1 + \frac{3^{5/3} - 3^{7/3}}{8}
  \approx 0.157501,
\end{align*}
attained when~$t$ is equal to
\begin{align*}
  \frac{2\cdot 3^{2/3} - 3^{1/3} - 2}{5} \approx 0.143584.
\end{align*}

We conjecture that this is actually the maximum value of~$\phi(W_4)$
for~$\phi\in\Hom^+(\cA^0,\RR)$.

\begin{conjecture}
  In the theory of tournaments, we have
  \begin{align*}
    \max\{\phi(W_4) : \phi\in\Hom^+(\cA^0,\RR)\} & = 1 + \frac{3^{5/3} - 3^{7/3}}{8}.
  \end{align*}
\end{conjecture}

Using the flag~algebra semidefinite method, we were able to obtain the bound
\begin{align*}
  \forall \phi\in\Hom^+(\cA^0,\RR), \phi(W_4) \leq 0.157516,
\end{align*}
subject to floating point rounding errors. This is suggests that the conjecture is
true and that there may be a straightforward (but numerically intensive) proof using
the semidefinite method and rounding techniques
(see~\cite{BHLPUV:MinimumNumberOfMonotoneSubsequences,
  CKPSTY:MonochromaticTrianglesInThreeColouredGraphs,
  DHMNS:AProblemOfErdosOnTheMinimumNumberOfkCliques,
  FV:ApplicationsOfTheSemidefiniteMethod, PV:MinimumNumberOfkCliques} for some
examples).

The intuition of the recursive construction of~$\bm{S_{N,t}}$ is that at every step
we have one part~$A_{i-1}\setminus A_i$ that maximizes the density of~$\vec C_3$
(hence is almost balanced) and another part~$A_i$ whose vertices all beat the
first part. This maximizes the occurrences of~$W_4$ with exactly one
vertex in the latter part, and since only one vertex is being selected in it, we
might as well repeat this structure recursively in~$A_i$.

In this particular construction, we chose the almost balanced part to be
quasi-random. However, one might wonder if this is the best we can do in the class of
almost balanced tournaments to maximize the density of~$W_4$, but the following
couple of lemmas show that this is indeed the case.

\begin{lemma}\label{lem:W4L4}
  In the theory of tournaments, if~$\phi\in\Hom^+(\cA^0,\RR)$ is balanced,
  then~$\phi(W_4)=\phi(L_4)$.
\end{lemma}

\begin{proof}
  Since~$\phi$ is balanced, we have~$\bm{\phi^1}(\alpha) = \bm{\phi^1}(\beta)$
  a.s. In particular, this means that
  \begin{align*}
    \frac{\phi(\Tr_4 + W_4)}{4} & = \expect{\bm{\phi^1}(\alpha)^3} =
    \expect{\bm{\phi^1}(\beta)^3} = \frac{\phi(\Tr_4 + L_4)}{4},
  \end{align*}
  hence~$\phi(W_4) = \phi(L_4)$.
\end{proof}

\begin{lemma}
  In the theory of tournaments, if~$\phi\in\Hom^+(\cA^0,\RR)$ is balanced,
  then~$\phi(W_4)\leq 1/8$ with equality if and only if~$\phi$ is the quasi-random
  tournament~$\phiqr$.
\end{lemma}

\begin{proof}
  Property~$P_2$ of Chung--Graham~\cite{CG:QuasiRandomTournaments} says\footnote{In
    their paper, Chung and Graham work with labelled non-induced densities instead of
    unlabelled induced densities, so a straightforward translation is necessary to get
    this value.} that if~$\psi\in\Hom^+(\cA^0,\RR)$, then~$\psi(\Tr_4+R_4)\geq 3/4$
  with equality if and only if~$\psi=\phiqr$, hence~$\psi(W_4+L_4)\leq 1/4$ with
  equality if and only if~$\psi=\phiqr$.

  On the other hand, since~$\phi$ is balanced, Lemma~\ref{lem:W4L4}
  implies that~$\phi(W_4) = (\phi(W_4+L_4))/2 \leq 1/8$.

  Since~$\phiqr$ is also balanced, the result follows.
\end{proof}

\bigskip

Focusing back on the carousel homomorphism, as we mentioned on
Remark~\ref{rmk:Rname}, the choice of the notation~$R_{2n+1}$ comes from the
similarity of the structure of these tournaments with the structure of~$R_4$. Given
this structural similarity, the following conjecture is natural.

\begin{conjecture}
  For every~$n\in\NN$, the carousel homomorphism~$\phi_R$
  maximizes the density of~$R_{2n+1}$, that is, we have
  \begin{align*}
    \max\{\phi(R_{2n+1}) : \phi\in\Hom^+(\cA^0,\RR)\} & = \phi_R(R_{2n+1}).
  \end{align*}
\end{conjecture}

And if the above conjecture is true, then naturally the following conjecture arises.
\begin{conjecture}
  For every~$n\geq 2$, a homomorphism~$\phi\in\Hom^+(\cA^0,\RR)$ maximizes the
  density of~$R_{2n+1}$ if and only if~$\phi=\phi_R$.
\end{conjecture}

\section*{Acknowledgement}
I am grateful to Alexander Razborov for helpful comments.

\bibliographystyle{alpha}
\bibliography{refs}
\end{document}